\documentclass[preprint,12pt]{elsarticle}

\usepackage{algorithmic}
\usepackage{algorithm}
\usepackage{graphicx}
\usepackage{amsmath,amsfonts,amssymb}
\usepackage{epsfig,makecell,float}
\usepackage{setspace,mathrsfs}
\usepackage{tocloft}
\usepackage{textcomp}
\usepackage{multirow,indentfirst,times,color}
\usepackage[caption=false,font=footnotesize]{subfig}
\usepackage{booktabs}
\usepackage{threeparttable}
\usepackage{amsthm}
\usepackage{extarrows}
\usepackage{bm}

\renewcommand{\arraystretch}{1.5}
\usepackage[titletoc]{appendix}
\usepackage{epstopdf}
\usepackage{natbib}

\usepackage{todonotes}
\usepackage{hyperref}

\biboptions{numbers,sort&compress}
\newtheorem{thm}{Theorem}
\newtheorem{defn}{Definition}

\newtheorem{lem}{Lemma}
\newtheorem{corollary}{Corollary}

\begin{document}

\begin{frontmatter}



\title{The uncertainty principles of random signals related to the  linear canonical transform} 

\author[a,b]{Jia-Yin Peng} 
\author[a,b]{Bing-Zhao Li\corref{mycorrespondingauthor}}
\cortext[mycorrespondingauthor]{Corresponding author}\ead{li\_bingzhao@bit.edu.cn}

\affiliation[a]{organization={School of Mathematics and Statistics, Beijing Institute of Technology},
	city={Beijing},
	postcode={100081}, 
	country={China}}

\affiliation[b]{organization={Beijing Key Laboratory on MCAACI, Beijing Institute of Technology},
	city={Beijing},
	postcode={100081}, 
	country={China}}

\begin{abstract}
	In this paper, we investigate uncertainty principles for random signals associated with the linear canonical transform (LCT). First, the LCT of random signals is formulated on the probability space. Based on this representation, the Heisenberg uncertainty principle is established to characterize the relationship between the expectations in the time and frequency domains. Furthermore, the Donoho-Stark uncertainty principle, developed from a measure theoretic perspective, reveals that a random signal cannot be simultaneously concentrated on arbitrarily small sets in both the time and frequency domains. The bounds obtained in these two uncertainty principles explicitly depend on the LCT parameters, indicating that the LCT offers greater flexibility than the Fourier transform (FT). The corresponding results in the fractional Fourier transform and FT domains are also given as special cases.
\end{abstract}




\end{frontmatter}

\section{Introduction}
Uncertainty principles constitute an  important topic in harmonic analysis and were originally rooted in quantum mechanics. The celebrated Heisenberg uncertainty principle (HUP), proposed by the German physicist Heisenberg, states that the position and momentum of an electron cannot be measured simultaneously with arbitrary precision \cite{heisenberg}. From the perspective of signal processing, the HUP reveals a fundamental trade-off between the energy distribution of a signal in the time and frequency domains \cite{2026Buncertainty}. In mathematical terms, for a signal $f(t)$ and its Fourier transform (FT) $\hat{f}(\xi)$, the time duration $\Delta f(t)$ and the frequency bandwidth $\Delta \hat{f}(\xi)$ satisfy the following inequality \cite{HUN}
\begin{equation}
\Delta f(t)\Delta \hat{f}(\xi) \geq \frac{1}{2}.
\label{eq:heisenberg}
\end{equation}
This principle has inspired extensive subsequent studies on uncertainty principles and serves as a significant tool in time-frequency analysis \cite{zhang2025H, hw, FrFtLp}. As a representative development, Donoho and Stark studied time and frequency limitations from the viewpoint of measurable sets \cite{dsu}. Their result shows that a nonzero signal and its FT cannot both be concentrated on sets with arbitrarily small measures. 
More precisely, if a signal $f(t)$ is $\varepsilon_T$-concentrated on a measurable set $T$ and its FT $\hat{f}(\xi)$ is $\varepsilon_W$-concentrated on another measurable set $W$, with $\varepsilon_T,\varepsilon_W\geq 0$ and $\varepsilon_T+\varepsilon_W<1$, then \cite{DStwo}
\begin{equation}
|T||W|
\geq
\left(1-\varepsilon_T-\varepsilon_W\right)^2 .
\label{DS}
\end{equation}
The Donoho-Stark uncertainty principle (DSUP) provides a useful theoretical tool for signal recovery from incomplete, missing and bandlimited observations \cite{DSE, dslp, tuan2010donoho}.

The above uncertainty principles and their extensions have led to a substantial body of theoretical results, most of which have been developed for deterministic signals that depend only on the time variable $t$ \cite{uncertaintyhj, uwlct}. However, real signals are often accompanied by random effects caused by measurements, environments or perturbations. 
Therefore, the study of random signals is of significant practical importance in signal processing \cite{randomqt, frftrandom}. 
In general, a random signal can be represented as $f(t,\eta)$, where the random parameter $\eta$ is defined on the probability space. 
Since $f(t,\eta)$ reduces to a deterministic signal for each fixed $\eta$, random signals can be regarded as a generalization of deterministic signals \cite{RandomFrFt, xurandom}.

In recent years, uncertainty principles for random signals have begun to attract increasing attention. Dang et al. investigated HUP and DSUP for random signals \cite{dang2023}. More precisely, they established a HUP for random signals in the FT domain, which can be expressed as
\begin{equation}
\mathbb{E}_{\eta}\!\left[\sigma_t^2(\eta)\right]
\mathbb{E}_{\eta}\!\left[\sigma_{\xi}^2(\eta)\right]
\geq
\frac{1}{4}
+
\left[
\int_{\Omega}\int_{-\infty}^{+\infty}
\left|
(t-\mu_t)
\left(
\frac{\partial\theta(t,\eta)}{\partial t}
-
\mu_{\xi}
\right)
\right|
\psi^2(t,\eta)\,dt\,dP
\right]^2,
\label{HUPFT}
\end{equation}
where
\begin{equation}
\mathbb{E}_{\eta}\!\left[\sigma_t^2(\eta)\right]
=
\int_{\Omega}\int_{-\infty}^{+\infty}
(t-\mu_t)^2 |f(t,\eta)|^2\,dt\,dP,
\label{et}
\end{equation}
and
\begin{equation}
\mathbb{E}_{\eta}\!\left[\sigma_{\xi}^2(\eta)\right]
=
\int_{\Omega}\int_{-\infty}^{+\infty}
(\xi-\mu_{\xi})^2 |\hat{f}(\xi,\eta)|^2\,d\xi\,dP.
\label{ef}
\end{equation}
Here, $\mu_t$ and $\mu_{\xi}$ denote the mean time and the mean Fourier frequency, respectively, given by
\begin{equation}
\mu_t
=
\int_{\Omega}\int_{-\infty}^{+\infty}
t |f(t,\eta)|^2\,dt\,dP,
\label{mt}
\end{equation}
and
\begin{equation}
\mu_{\xi}
=
\int_{\Omega}\int_{-\infty}^{+\infty}
\xi |\hat{f}(\xi,\eta)|^2\,d\xi\,dP.
\label{mf}
\end{equation}
In addition, they established a DSUP for random signals in the FT domain. More precisely, let $(\Omega,\mathcal{F},P)$ be a probability space, $B_T\times T$ and $B_W\times W$ are measurable sets in $\Omega \times \mathbb{R}$. 
If $f(t,\eta)\in L^2(\mathbb{R}\times\Omega)$ is $\varepsilon_{B_T\times T}$-concentrated on $B_T\times T$ and $\hat f(\xi,\eta)$ is $\varepsilon_{B_W\times W}$-concentrated on $B_W\times W$, then
\begin{equation}
|T||W|\,P(B_T\cap B_W)
\geq
\left(
1-\varepsilon_{B_T\times T}
-\varepsilon_{B_W\times W}
\right)^2 .
\label{DSFT}
\end{equation}
Here, $|T|$ and $|W|$ denote the Lebesgue measures of $T$ and $W$, respectively, while $P(B_T\cap B_W)$ describes the probability that the corresponding time domain and frequency domain concentration events occur simultaneously. This result shows that the simultaneous concentration of a random signal and its FT is constrained not only by the measures of the time and frequency sets, but also by the probability of the associated random events.

Motivated by the above developments, this paper focuses on uncertainty principles for random signals in the linear canonical transform (LCT) \cite{LCT} . As a generalization of the FT and the fractional Fourier transform (FrFT), the LCT offers greater flexibility through its parameter selection and serves as a powerful tool for nonstationary signal analysis \cite{2026Suncertainty, xulctup, peng20252p}. Therefore, the study of random signals in the LCT domain is of considerable importance for analyzing signals involving both nonstationarity and random variations in practical signal processing \cite{mairandomlct}.

The remainder of this paper is organized as follows. Section \ref{sec2} introduces the basic definitions required for the subsequent analysis. Section \ref{sec3} establishes the HUP for random signals in the LCT domain. Section \ref{sec4} further investigates the DSUP for random signals in the LCT domain. Finally, Section \ref{sec5} concludes the paper.

\section{Preliminaries}
\label{sec2}
In this section, we review the mathematical representation of random signals and their Fourier transform (FT), and extend the fractional Fourier transform (FrFT) and the linear canonical transform (LCT) to random signals.

\begin{defn}
Let $(\Omega,\mathcal{F},P)$ be a probability space, where $\Omega$ represents the sample set, $\mathcal{F}$ is a $\sigma$-algebra and $P$ denotes the probability measure. A random signal is represented as \cite{dang2023}
\begin{equation}
f(t,\eta)=\psi(t,\eta)\mathrm{e}^{\mathrm{j}\theta(t,\eta)}, \ \ t \in \mathbb{R}, \ \ \eta \in \Omega,
\end{equation}
where $\psi(t,\eta)$ and $\theta(t,\eta)$ denote the amplitude and phase of the random signal, respectively, with $t$ being the time variable and $\eta$ representing the random parameter.
\end{defn}
Assume that $f(\cdot, \eta) \in L^2(\mathbb{R})$ for almost every $\eta \in \Omega$. Then for each fixed $\eta$, $f(\cdot,\eta)$ can be regarded as a deterministic signal on $\mathbb{R}$. Its energy is given by
\begin{equation}
\|f(\cdot,\eta)\|_{L^2(\mathbb{R})}^{2}
=
\int_{\mathbb{R}} |f(t,\eta)|^{2}\,dt .
\end{equation}
A random signal is said to be square-integrable if its mean energy is finite, 
\begin{equation}
\|f\|_{L^2(\mathbb{R}\times\Omega)}^{2}
=
\int_{\Omega}\int_{\mathbb{R}} |f(t,\eta)|^{2}\,dt\,dP
=
\mathbb{E}_{\eta}\left[
\|f(\cdot,\eta)\|_{L^2(\mathbb{R})}^{2}
\right]
<\infty .
\end{equation}
In what follows, we assume that the random signals under consideration belong to $L^2(\mathbb{R} \times \Omega)$ and are normalized such that
\begin{equation}
\|f\|_{L^2(\mathbb{R}\times\Omega)}^{2}=1 .
\end{equation}
\begin{defn}
	\label{FT}
The FT of the random signal $f(t,\eta)$ is defined by \cite{dang2023}
\begin{equation}
\hat{f}(\xi,\eta)
=
\frac{1}{\sqrt{2\pi}}
\int_{-\infty}^{+\infty}
f(t,\eta)\mathrm{e}^{-\mathrm{j}t\xi}\,dt.
\label{eq:random_signal_FT}
\end{equation}
\end{defn}

We extend the Definition~\ref{FT} to the FrFT and the LCT, leading to Definition~\ref{FrFTrm} and \ref{LCT}.
\begin{defn}
	\label{FrFTrm}
The FrFT of the random signal $f(t,\eta)$ is defined by
\begin{equation}
F_{\alpha}(\rho,\eta)
=
\frac{1}{\sqrt{\mathrm{j}2\pi\sin\alpha}}
\int_{-\infty}^{+\infty}
f(t,\eta)
\mathrm{e}^{
	\mathrm{j}
	\frac{(t^2+\rho^2)\cos\alpha-2t\rho}{2\sin\alpha}}
\,dt, \ \ \sin\alpha\neq 0.
\label{eq:random_signal_FrFT}
\end{equation}
where $\alpha$ is the rotation parameter.
\end{defn}

\begin{defn}
	\label{LCT}
The LCT of the random signal $f(t,\eta)$ is defined by
\begin{equation}
\label{eq:LCT_random}
L_A(u,\eta)
=
\mathcal{L}_A[f(t,\eta)](u)
=
\begin{cases}
\displaystyle
\int_{-\infty}^{+\infty}
f(t,\eta)K_A(t,u)\,dt, & b\neq 0, \\[2ex]
\displaystyle
\sqrt{d} \mathrm{e}^{\mathrm{j}\frac{cd u^2}{2}}
f(du,\eta), & b=0,
\end{cases}
\end{equation}
where the parameter matrix $A=\left[
\renewcommand{\arraystretch}{0.85}
\begin{array}{@{}cc@{}}
a & b \\
c & d
\end{array}
\right]$, 
with $a,b,c,d\in\mathbb{R}, ad-bc=1$, and the LCT kernel $K_A(t,u)$ is given by
\begin{equation}
\label{eq:LCT_kernel_random}
K_A(t,u)
=
\frac{1}{\sqrt{\mathrm{j}2\pi b}}\mathrm{e}^{\mathrm{j} \left( \frac{a}{2b}t^2 -\frac{1}{b}tu +\frac{d}{2b}u^2 \right)}.
\end{equation}
\end{defn}
It can be seen from the above definition, when $A=\left[
\renewcommand{\arraystretch}{0.85}
\begin{array}{@{}cc@{}}
\cos\alpha & \sin\alpha \\
-\sin\alpha & \cos\alpha
\end{array}
\right]$
the LCT reduces to the FrFT for random signals and
when
$A=\left[
\renewcommand{\arraystretch}{0.85}
\begin{array}{@{}cc@{}}
0 & 1 \\
-1 & 0
\end{array}
\right]$
the LCT reduces to the FT for random signals.

\section{HUP for random signals in the LCT domain}
\label{sec3}
In this section, we introduce the mean and expectation of random signals in the linear canonical transform (LCT) domain and establish the corresponding Heisenberg uncertainty principle (HUP).
\begin{defn}
The mean of LCT frequency is defined as
\begin{equation}
\mu_u
=
\int_{\Omega}\int_{-\infty}^{+\infty}
u |L_A(u,\eta)|^2\,du\,dP,
\label{mlct}
\end{equation}
and the duration is given by
\begin{equation}
\label{3.3}
\sigma_u^2(\eta)
=
\int_{-\infty}^{+\infty}
(u-\mu_u)^2 |L_A(u,\eta)|^2\,du.
\end{equation}
\end{defn}
\begin{lem}
	\label{lem1}
Suppose that for almost every $\eta\in\Omega$,
$f(t,\eta)$ , $tf(t,\eta)$, $\partial_t f(t,\eta) \in L^2(\mathbb{R})$, $\psi(\cdot, \eta) $, $\theta(\cdot, \eta)$ are continuously differentiable, and $\displaystyle \lim_{|t|\to\infty} f(t,\eta)=0$. Then, the following identity holds
\begin{equation}
\label{lemeq1}
\begin{aligned}
\sigma_u^2(\eta)
&=
\int_{-\infty}^{+\infty}
\left(
at+b\partial_t\theta(t,\eta)-\mu_u
\right)^2
\psi^2(t,\eta)\,dt  
+
b^2\int_{-\infty}^{+\infty}
\left|
\partial_t\psi(t,\eta)
\right|^2\,dt .
\end{aligned}
\end{equation}
where $a$ and $b$ are the LCT parameters associated with the matrix
$A$.
\end{lem}

\begin{proof}
We first compute the  mean $\mu_u$ in the LCT domain,
	\begin{equation}
	\label{3.1}
	\begin{aligned}
	\mu_u
	&=
	\int_{\Omega}\int_{-\infty}^{+\infty}
	u|L_A(u,\eta)|^2\,du\,dP\\
	&=	\int_{\Omega}\int_{-\infty}^{+\infty}
	\overline{L_A(u,\eta)}	uL_A(u,\eta)\,du\,dP\\
	&=		\int_{\Omega} \int_{-\infty}^{+\infty} 	\overline{L_A(u,\eta)}
		\int_{-\infty}^{+\infty} f(t,\eta)uK_A(t,u)\,dt  \,du \,dP.\\
	\end{aligned}
	\end{equation}
	Using the differentiation property of the LCT, we have
	\begin{equation}
	\label{3.2d}
	uK_A(t,u)
	=
	atK_A(t,u)
	-
	\frac{b}{\mathrm{j}}
	\partial_t K_A(t,u).
	\end{equation}
	It follows from \eqref{3.1} and \eqref{3.2d} that
	\begin{equation}
	\begin{aligned}
	\mu_u
	&=		\int_{\Omega} \int_{-\infty}^{+\infty} 	\overline{L_A(u,\eta)}
\int_{-\infty}^{+\infty} f(t,\eta)uK_A(t,u)\,dt  \,du \,dP\\
	&=\int_{\Omega} \int_{-\infty}^{+\infty}	\overline {L_A(u,\eta)} 	\left( a\int_{-\infty}^{+\infty}
	t f(t,\eta)K_A(t,u)\,dt
	-
	\frac{b}{\mathrm{j}}
	\int_{-\infty}^{+\infty}
	f(t,\eta) \partial_t K_A(t,u)\,dt \right)\,du \,dP\\
	&=\int_{\Omega} \int_{-\infty}^{+\infty}	\overline{L_A(u,\eta)} 	\left(a \mathcal{L}_{A} [ t f(t,\eta) ] (u)-\frac{b}{\mathrm{j}} \int_{-\infty}^{+\infty}f(t,\eta) \,d(K_A(t,u))\,dt \right)\,du \,dP\\
	&=
	\int_{\Omega} \int_{-\infty}^{+\infty} \overline{L_A(u,\eta)} 
	\Biggl(
	a\,\mathcal{L}_{A}[t f(t,\eta)](u)
	-
	\frac{b}{\mathrm{j}}
	\Bigl(
	f(t,\eta) K_A(t,u)\big|_{-\infty}^{+\infty} \\
	&\qquad\qquad\qquad\qquad\qquad
	-
	\int_{-\infty}^{+\infty} K_A(t,u) \partial_t f(t,\eta)\,dt
	\Bigr)
	\Biggr)
	\,du\,dP.
	\end{aligned}
	\end{equation}
	By the assumption that $f(t,\eta)$ vanishes at infinity, we obtain
	\begin{equation}
	\begin{aligned}
		\mu_u
	&=\int_{\Omega} \int_{-\infty}^{+\infty}	\overline{L_A(u,\eta)} 	\left(a \mathcal{L}_{A} [t f(t,\eta)] (u)+\frac{b}{\mathrm{j}}  \int_{-\infty}^{+\infty}  K_A(t,u) \partial_t f(t,\eta)\,dt \right)\,du \,dP\\
	&=\int_{\Omega} \int_{-\infty}^{+\infty}	\overline{L_A(u,\eta)} 	\left(a \mathcal{L}_{A} [t f(t,\eta)] (u)+\frac{b}{\mathrm{j}} \mathcal{L}_{A} [\partial_t f(t,\eta)]  (u) \right)\,du \,dP.
	\end{aligned}
	\end{equation}
	Using the Parseval theorem in the LCT domain, we get
	\begin{equation}
	\label{mu_u}
	\begin{aligned}
	\mu_u
	&=
	a\int_{\Omega}\int_{-\infty}^{+\infty}
	\overline{f(t,\eta)}\,t f(t,\eta)\,dt\,dP
	+
	\frac{b}{\mathrm{j}}
	\int_{\Omega}\int_{-\infty}^{+\infty}
	\overline{f(t,\eta)}\,\partial_t f(t,\eta)\,dt\,dP \\
	&=
	a\int_{\Omega}\int_{-\infty}^{+\infty}
	t|f(t,\eta)|^2\,dt\,dP
	+
	\frac{b}{\mathrm{j}}
	\int_{\Omega}\int_{-\infty}^{+\infty}
	\psi(t,\eta)e^{-j\theta(t,\eta)}
	\partial_t\!\left(
	\psi(t,\eta)e^{\mathrm{j}\theta(t,\eta)}
	\right)
	dt\,dP \\
	&=
	a\int_{\Omega}\int_{-\infty}^{+\infty}
	t|f(t,\eta)|^2\,dt\,dP
	+
	\frac{b}{\mathrm{j}}
	\int_{\Omega}\int_{-\infty}^{+\infty}
	\left(
	\psi(t,\eta)\partial_t\psi(t,\eta)
	+
	\mathrm{j} \psi^2(t,\eta)\partial_t\theta(t,\eta)
	\right)
	dt\,dP.
	\end{aligned}
	\end{equation}
	We then compute $\sigma_u^2(\eta)$ in the LCT domain. By \eqref{3.3} and \eqref{mu_u} , we have
	\begin{equation}
	\begin{aligned}
	\sigma_u^2(\eta)
	&= \int_{-\infty}^{+\infty} 
	(u-\mu_u)^2 
	\left| L_A(u,\eta) \right|^2 \, du \\
	&= \int_{-\infty}^{+\infty} \left| u L_A(u,\eta)- \mu_u L_A(u,\eta) \right|^2 du \\
	&= \int_{-\infty}^{+\infty} 
	\left| 
	a\mathcal{L}_A[t f(t,\eta)](u) 
	+ \frac{b}{\mathrm{j}}
	\mathcal{L}_A\left[
	\partial_t f(t,\eta)
	\right](u)
	- \mu_u \mathcal{L}_A[f(t,\eta)](u)
	\right|^2 du \\
	&= \int_{-\infty}^{+\infty} 
	\left| 
	\mathcal{L}_A
	\left[
	\left(
	at+\frac{b}{\mathrm{j}}\partial t-\mu_u
	\right)
	f(t,\eta)
	\right](u)
	\right|^2 du .
	\end{aligned}
	\end{equation}
	Using the Parseval theorem, we get
	\begin{equation}
	\begin{aligned}
	\sigma_u^2(\eta) 
	&= \int_{-\infty}^{+\infty} \left| \left( at + \frac{b}{\mathrm{j}}\partial t - \mu_u \right)f(t,\eta) \right|^2 dt \\
	&= \int_{-\infty}^{+\infty} \left| at \psi(t,\eta)e^{\mathrm{j}\theta(t,\eta)} + \frac{b}{\mathrm{j}}\partial_t\left(\psi(t,\eta)e^{\mathrm{j}\theta(t,\eta)}\right) - \mu_u \psi(t,\eta)e^{\mathrm{j}\theta(t,\eta)} \right|^2 dt \\
	&= \int_{-\infty}^{+\infty} \left| a t \psi(t,\eta) e^{\mathrm{j}\theta(t,\eta)} + \frac{b}{\mathrm{j}}\left( \partial_t \psi(t,\eta) e^{\mathrm{j}\theta(t,\eta)} + \psi(t,\eta) e^{\mathrm{j}\theta(t,\eta)} \mathrm{j}\partial_t \theta(t,\eta) \right) - \mu_u \psi(t,\eta) e^{\mathrm{j}\theta(t,\eta)} \right|^2 dt \\
	&= \int_{-\infty}^{+\infty} \left| \left( a t + b \partial_t \theta(t,\eta)  - \mu_u \right) \psi(t,\eta) + \frac{b}{\mathrm{j}}  \partial_t \psi(t,\eta)  \right|^2 dt \\
	&=  \int_{-\infty}^{+\infty} \left( at + b \partial_t \theta(t,\eta) - \mu_u \right)^2 \psi^2(t,\eta) dt + b^2 \int_{-\infty}^{+\infty} \left| \partial_t \psi(t,\eta) \right|^2 dt.
	\end{aligned}
	\end{equation}
	This completes the proof.
\end{proof}
	Lemma~\ref{lem1} establishes the duration of random signals in the LCT domain. 
	By choosing the LCT parameter matrix corresponding to the fractional Fourier transform (FrFT), we obtain the following corollary.
\begin{corollary}
	\label{cor1}
	Suppose that for almost every $\eta\in\Omega$, $f(t,\eta)$, $tf(t,\eta)$, $\partial_t f(t,\eta) \in L^2(\mathbb{R})$, $\psi(\cdot,\eta)$ and $\theta(\cdot,\eta)$ are continuously differentiable, and $\displaystyle \lim_{|t|\to\infty} f(t,\eta)=0$.
	For the FrFT of rotation parameter $\alpha$ with $\sin\alpha\neq 0$, the following identity holds
		\begin{equation}
		\begin{aligned}
		\sigma_{\rho}^2(\eta)
		&=
		\int_{-\infty}^{+\infty}
		\left(
		t\cos\alpha
		+
		\sin\alpha\,\partial_t \theta(t,\eta)
		-
		\mu_{\rho}
		\right)^2
		\psi^2(t,\eta)\,dt  \\
		&\quad
		+
		\sin^2\alpha
		\int_{-\infty}^{+\infty}
		\left| \partial_t \psi(t,\eta)
		\right|^2\,dt .
		\end{aligned}
		\end{equation}
		where the mean of FrFT frequency is defined as
		\begin{equation}
		\mu_{\rho}
		=
		\int_{\Omega}\int_{-\infty}^{+\infty}
		\rho |F_{\alpha}(\rho,\eta)|^2\,d\rho\,dP.
		\end{equation}
\end{corollary}
Moreover, Lemma~\ref{lem1} reduces to the corresponding result in the Fourier domain, which agrees with~\cite{dang2023}.

\begin{thm}
\label{thm1}
Let the assumptions of Lemma~\ref{lem1} hold. 
Then the HUP for random signals in the LCT domain is stated as follows
	\begin{equation}
	\begin{aligned}
	\mathbb{E}_{\eta}[\sigma_t^2(\eta)]
	\mathbb{E}_{\eta}[\sigma_u^2(\eta)]
	\geq
	\frac{b^2}{4}
	+
	\left[
	\int_{\Omega}\int_{-\infty}^{+\infty}
	\left|
	(t-\mu_t)
	\left(
	at+b\partial_t\theta(t,\eta)-\mu_u
	\right)
	\right|
	\psi^2(t,\eta)\,dt\,dP
	\right]^2 .
	\end{aligned}
	\end{equation}
	where
		\begin{equation}
	\mathbb{E}_{\eta}[\sigma_t^2(\eta)]
	=
	\int_{\Omega}\int_{-\infty}^{+\infty}
	(t-\mu_t)^2|f(t,\eta)|^2\,dt\,dP,
	\end{equation}
	and
	\begin{equation}
	\label{LCTE}
	\mathbb{E}_{\eta}[\sigma_u^2(\eta)]
	=
	\int_{\Omega} \int_{-\infty}^{+\infty} 
	(u-\mu_u)^2 
	\left| L_A(u,\eta) \right|^2 \, du \,dP.
	\end{equation}
\end{thm}

\begin{proof}
	Using \eqref{lemeq1}, we first calculate \eqref{LCTE} and obtain 
	\begin{equation}
	\begin{aligned}
	\mathbb{E}_{\eta}[\sigma_u^2(\eta)]
	&=	\int_{\Omega}\int_{-\infty}^{+\infty}
	\left(
	at+b\partial_t\theta(t,\eta)-\mu_u
	\right)^2
	\psi^2(t,\eta)\,dt\,dP\\
	&\quad+b^2
	\int_{\Omega}\int_{-\infty}^{+\infty}
	|\partial_t\psi(t,\eta)|^2\,dt\,dP.
	\end{aligned}
	\end{equation}
    Next, we prove the following two inequalities separately,
    \begin{equation}
    \label{thm2.1}
	b^2 \int_{\Omega}\int_{-\infty}^{+\infty}
	|\partial_t\psi(t,\eta)|^2\,dt\,dP
	\int_{\Omega}\int_{-\infty}^{+\infty}
	(t-\mu_t)^2\psi^2(t,\eta)\,dt\,dP
	\geq \frac{b^2}{4},
    \end{equation}
	and
    \begin{equation}
    \label{thm2.2}
	\begin{aligned}
	&\int_{\Omega}\int_{-\infty}^{+\infty}
	\left(
	at+b\partial_t\theta(t,\eta)-\mu_u
	\right)^2
	\psi^2(t,\eta)\,dt\,dP
	\int_{\Omega}\int_{-\infty}^{+\infty}
	(t-\mu_t)^2\psi^2(t,\eta)\,dt\,dP \\
	&\geq
	\left[
	\int_{\Omega}\int_{-\infty}^{+\infty}
	\left|
	(t-\mu_t)
	\left(
	at+b\partial_t\theta(t,\eta)-\mu_u
	\right)
	\right|
	\psi^2(t,\eta)\,dt\,dP
	\right]^2 .
	\end{aligned}
	\end{equation}
	For \eqref{thm2.1}, by applying the Cauchy-Schwarz inequality, we obtain
	\begin{equation}
	\begin{aligned}
	&b^2 \int_{\Omega}\int_{-\infty}^{+\infty}
	|\partial_t\psi(t,\eta)|^2\,dt\,dP
	\int_{\Omega}\int_{-\infty}^{+\infty}
	(t-\mu_t)^2\psi^2(t,\eta)\,dt\,dP\\
	& \ge b^2 \left[
	\int_{\Omega}\int_{-\infty}^{+\infty}
	\left| \partial_t\psi(t,\eta) 
	(t-\mu_t) \psi(t,\eta)
	\right|\,dt\,dP
	\right]^2 \\
	& \ge b^2 \left[
	\int_{\Omega}\int_{-\infty}^{+\infty} \partial_t\psi(t,\eta) 
	(t-\mu_t) \psi(t,\eta) \,dt\,dP
	\right]^2 \\
	&=  \frac{b^2}{4} \left[
	\int_{\Omega} \left(
	(t-\mu_t) \psi^2(t,\eta) \big|^{+\infty}_{-\infty} -\int_{-\infty}^{+\infty} \psi^2(t,\eta) \,dt\right) \,dP
	\right]^2\\
	& =  \frac{b^2}{4} \left[
	\int_{\Omega} \int_{-\infty}^{+\infty} \psi^2(t,\eta) \,dt \,dP
	\right]^2\\
	&= \frac{b^2}{4}.
	\end{aligned}
	\end{equation}
	Then we prove \eqref{thm2.2}. Similarly, by the Cauchy-Schwarz inequality,
\begin{equation}
	\begin{aligned}
	&\left[
	\int_{\Omega}\int_{-\infty}^{+\infty}
	(t-\mu_t)^2\psi^2(t,\eta)\,dt\,dP
	\right]
	\left[
	\int_{\Omega}\int_{-\infty}^{+\infty}
	\left(
	at+b\partial_t\theta(t,\eta)-\mu_u
	\right)^2
	\psi^2(t,\eta)\,dt\,dP
	\right] \\
	&\ge \left[
	\int_{\Omega}\int_{-\infty}^{+\infty}
	\left|
	(t-\mu_t)\psi(t,\eta)
	\left(
	at+b\partial_t\theta(t,\eta)-\mu_u
	\right)
	\psi(t,\eta)	\right| \,dt\,dP
	\right]^2 \\
	&=
	\left[
	\int_{\Omega}\int_{-\infty}^{+\infty}
	\left|
	(t-\mu_t)
	\left(
	at+b\partial_t\theta(t,\eta)-\mu_u
	\right)
	\right|
	\psi^2(t,\eta)\,dt\,dP
	\right]^2 .
	\end{aligned}
\end{equation}
	Combining the above estimates, we get
	\[
	\begin{aligned}
	\mathbb{E}_{\eta}[\sigma_t^2(\eta)]
	\mathbb{E}_{\eta}[\sigma_u^2(\eta)]
	&\geq
	\frac{b^2}{4}
	+
	\left[
	\int_{\Omega}\int_{-\infty}^{+\infty}
	\left|
	(t-\mu_t)
	\left(
	at+b\partial_t\theta(t,\eta)-\mu_u
	\right)
	\right|
	\psi^2(t,\eta)\,dt\,dP
	\right]^2 .
	\end{aligned}
	\]
	This completes the proof.
\end{proof}
\begin{corollary}
	Let the assumptions of Theorem~\ref{thm1} hold. 
Then the HUP for random signals in the FrFT domain is stated as follows
\begin{equation}
\begin{aligned}
&\mathbb{E}_{\eta}\!\left[\sigma_t^2(\eta)\right]
\mathbb{E}_{\eta}\!\left[\sigma_{\rho}^2(\eta)\right] \\
 \geq&
\frac{\sin^2\alpha}{4}
+
\Biggl[
\int_{\Omega}\int_{-\infty}^{+\infty}
\left|
(t-\mu_t)
\left(
t\cos\alpha
+\sin\alpha\,\partial_t\theta(t,\eta)
-\mu_{\rho}
\right)
\right|  
\psi^2(t,\eta)\,dt\,dP
\Biggr]^2 .
\end{aligned}
\end{equation}
where
\begin{equation}
\label{FrFT}
\mathbb{E}_{\eta}[\sigma_{\alpha}^2(\eta)]
=
\int_{\Omega} \int_{-\infty}^{+\infty} 
(\rho-\mu_{\alpha})^2 
\left| F_{\alpha}(\rho,\eta) \right|^2 \, d\rho \,dP.
\end{equation}
\end{corollary}

Moreover, the Fourier domain result in  \eqref{HUPFT} is obtained as a special case of Theorem~\ref{thm1}.

\section{DSUP for random signals in the LCT domain}
\label{sec4}
We begin by establishing several lemmas and operator-theoretic results. These preliminary results provide the foundation for deriving the Donoho-Stark uncertainty principle (DSUP) for random signals in the linear canonical transform (LCT) domain. We then use this principle to investigate the recovery of random signals from incomplete or partially missing observations.

A random signal $f(t,\eta)$ is said to be time-limited to a measurable set $B_T\times T$ if
\begin{equation}
f(t,\eta)=0,
\qquad \text{for a.e. }(\eta,t)\notin B_T\times T .
\end{equation}
Similarly, $f(t,\eta)$ is said to be frequency-limited to a measurable set $B_U\times U$ in the LCT domain if
\begin{equation}
L_A(u,\eta)=0,
\qquad \text{for a.e. }(\eta,u)\notin B_U\times U .
\end{equation}

Then the time-limiting operator is defined by
\begin{equation}
\label{tl}
\tau_{B_T \times T} f(t,\eta)
=
\chi_{B_T}(\eta)\chi_T(t)f(t,\eta),
\end{equation}
where $\chi_{B_T}$ and $\chi_T$ denote the characteristic functions of $B_T$ and $T$, respectively. And the frequency-domain limiting operator in the LCT domain is defined by
\begin{equation}
\label{fl}
\tau_{B_U \times U} f(t,\eta)
=
\chi_{B_U}(\eta)
\frac{1}{\sqrt{-\mathrm{j}2\pi b}}
\int_U L_A(u, \eta)
\mathrm{e}^{-\mathrm{j}\frac{a t^2+d u^2-2tu}{2b}}
\,du ,
\end{equation}
where $\chi_{B_U}$ denote the characteristic functions of $B_U$. Here, $\chi_{B_T}$ and $\chi_{B_U}$ are determined only by the measurable sets and are independent of the signal $f(t,\eta)$.

We shall also use the standard operator norm on $L^2(\mathbb{R}\times\Omega)$. 
For a bounded linear operator
$Q:L^2(\mathbb{R}\times\Omega)\to L^2(\mathbb{R}\times\Omega)$, its norm is defined by
\begin{equation}
\label{qf}
\|Q\|
=
\sup_{f\neq 0}
\frac{\|Qf\|_{L^2(\mathbb{R}\times\Omega)}}
{\|f\|_{L^2(\mathbb{R}\times\Omega)}} .
\end{equation}
In addition, if $Q$ acting on $f(t,\eta)$ can be expressed as
\begin{equation}
Qf(t,\eta)
=
\int_{\mathbb{R}} q(\eta,t,x)f(x,\eta)\,dx,
\end{equation}
where the kernel $q(\eta,t,x)$ is chosen such that
$P\left(
	\int_{\mathbb{R}}
	|f(x,\eta) q(\eta,t,x)|\,dx
	<\infty
	\right)=1, \ x\in\mathbb{R}$,
and
$\int_{\Omega}\int_{\mathbb{R}}
	|Qf(t,\eta)|^2\,dt\,dP
	<\infty$. Then its Hilbert-Schmidt norm is given by
\begin{equation}
\label{qhs}
\|Q\|_{\mathrm{HS}}
=
\left(
\int_{\Omega}
\int_{\mathbb{R}}
\int_{\mathbb{R}}
|q(\eta,t,x)|^2\,dx\,dt\,dP
\right)^{\frac{1}{2}}.
\end{equation}

 Next, we review the relationship between the operator norm of $Q$ and its Hilbert-Schmidt norm.
\begin{lem}
	\label{lem}
	If $q(\eta,t,x)$ in the operator $Q$ is independent of $f(t,\eta)$, then \cite{dang2023}
	\begin{equation}
	\|Q\|\leq \|Q\|_{HS}.
	\end{equation}
\end{lem}

The following lemma extends the Hilbert-Schmidt norm result of the limiting operator in Dang et al.~\cite{dang2023} to the LCT domain.

\begin{lem}
	\label{lem3}
	The Hilbert-Schmidt norm of $\tau_{B_U\times U}\tau_{B_T\times T}$ satisfies
	\begin{equation}
	\left\|
	\tau_{B_U\times U}\tau_{B_T\times T}
	\right\|_{HS}^2
	=
\frac{1}{2\pi |b|}
	|U||T|P(B_T\cap B_U), \ \ b \neq 0.
	\end{equation}
\end{lem}

\begin{proof}
	For any $f(t,\eta)\in L^2(\mathbb{R}\times\Omega)$, by \eqref{tl} and \eqref{fl}, we have
	\begin{equation}
	\begin{aligned}
	&\tau_{B_U \times U}\tau_{B_T \times T}f(t,\eta) \\
	=&
	\chi_{B_U}(\eta)
	\frac{1}{\sqrt{-\mathrm{j}2\pi b}}
	\int_U
	\mathcal{L}_A[\tau_{B_T \times T}f(t,\eta)](u)
	\mathrm{e}^{-\mathrm{j}\left(
		\frac{d}{2b}u^2-\frac{u}{b}t+\frac{a}{2b}t^2
		\right)}
	\,du \\[6pt]
	=&
	\chi_{B_U}(\eta)
	\frac{1}{\sqrt{-\mathrm{j}2\pi b}}
	\int_U
	\chi_{B_T}(\eta)
	\int_T
	f(x,\eta)
	\frac{1}{\sqrt{\mathrm{j}2\pi b}}
	\mathrm{e}^{\mathrm{j}\left(
		\frac{a}{2b}x^2-\frac{u}{b}x+\frac{d}{2b}u^2
		\right)}
	\mathrm{e}^{-\mathrm{j}\left(
		\frac{d}{2b}u^2-\frac{u}{b}t+\frac{a}{2b}t^2
		\right)}
	\,dx 	\,du \\
	=&
	\frac{1}{2\pi |b|}
	\chi_{B_T}(\eta)\chi_{B_U}(\eta)
	\int_U\int_T
	f(x,\eta)
	\mathrm{e}^{\mathrm{j}\left(
		\frac{a}{2b}(x^2-t^2)-\frac{u}{b}(x-t)
		\right)}
	\,dx\,du.
	\end{aligned}
	\end{equation}
 For convenience, we define
\begin{equation}
	h(\eta,t,x)
	=
	\frac{1}{2\pi |b|}
	\chi_{B_T}(\eta)\chi_{B_U}(\eta)
	\int_U
	\mathrm{e}^{\mathrm{j}\left(
		\frac{a}{2b}(x^2-t^2)-\frac{u}{b}(x-t)
		\right)}
	\,du.
\end{equation}
	It follows that
	\begin{equation}
\begin{aligned}
\tau_{B_U \times U}\tau_{B_T \times T}f(t,\eta)
&=
\int_T f(x,\eta)h(\eta,t,x)\,dx \\
&=
\int_{\mathbb{R}}\chi_T(x)f(x,\eta)h(\eta,t,x)\,dx.
\end{aligned}
\end{equation}
Then let
\begin{equation}
q(\eta,t,x)=\chi_T(x)h(\eta,t,x),
\end{equation}
so that
\begin{equation}
\tau_{B_U \times U}\tau_{B_T \times T}f(t,\eta)
=
\int_{\mathbb{R}}f(x,\eta)q(\eta,t,x)\,dx.
\end{equation}
Therefore, by \eqref{qhs}, we get
\begin{equation}
\begin{aligned}
\left\|
\tau_{B_U \times U}\tau_{B_T \times T}
\right\|_{HS}^2
&=
\int_{\Omega}\int_{\mathbb{R}}\int_{\mathbb{R}}
|q(\eta,t,x)|^2\,dx\,dt\,dP \\
&=
\int_{\Omega}\int_{\mathbb{R}}\int_T
|h(\eta,t,x)|^2\,dx\,dt\,dP.
\end{aligned}
\end{equation}

Moreover, by the Parseval theorem in the LCT domain, we obtain
\begin{equation}
\begin{aligned}
&\int_{\mathbb{R}}|h(\eta,t,x)|^2\,dt
=
\int_{\mathbb{R}}
\left|
\mathcal{L}[h(\eta,t,x)](v)
\right|^2
\,dv \\[6pt]
=& \int_{\mathbb{R}} \left| \int_{\mathbb{R}} \left(\frac{1}{2\pi |b|} \chi_{B_T}(\eta)\chi_{B_U}(\eta)
\int_U \mathrm{e}^{\mathrm{j}\left( \frac{a}{2b}(x^2-t^2) - \frac{u}{b}(x-t) \right)} du \,\right)
\frac{1}{\sqrt{\mathrm{j}2\pi b}} \mathrm{e}^{\mathrm{j}\left( \frac{a}{2b}t^2 - \frac{v}{b}t + \frac{d}{2b}v^2 \right)} dt \right|^2 dv \\[6pt]
= &\int_{\mathbb{R}} \left|  \frac{1}{2\pi |b|} \frac{1}{\sqrt{\mathrm{j}2\pi b}}  \chi_{B_T}(\eta)\chi_{B_U}(\eta)
\mathrm{e}^{\mathrm{j} \left(\frac{a}{2b}x^2+ \frac{d}{2b} v^2 \right) } \int_{\mathbb{R}} \int_U \mathrm{e}^{\mathrm{-j} \frac{u}{b}x }
\mathrm{e}^{\mathrm{j} \frac{u-v}{b}t } du \, dt \right|^2 dv \\[6pt]
=& \int_{\mathbb{R}} \left|  \frac{1}{2\pi |b|} \frac{1}{\sqrt{\mathrm{j}2\pi b}}  \chi_{B_T}(\eta)\chi_{B_U}(\eta)
\mathrm{e}^{\mathrm{j} \left(\frac{a}{2b}x^2+ \frac{d}{2b} v^2 \right) }  \int_U \mathrm{e}^{\mathrm{-j} \frac{u}{b}x } 2\pi |b|\,\delta(u-v)
du \,  \right|^2 dv \\[6pt]
=& \int_U \left| \frac{1}{\sqrt{\mathrm{j}2\pi b}} \chi_{B_T}(\eta)\chi_{B_U}(\eta)
\mathrm{e}^{\mathrm{j}\left( \frac{a}{2b}x^2 - \frac{u}{b}x + \frac{d}{2b}u^2 \right)} \right|^2 dv \\[6pt]
=& \frac{1}{2\pi |b|} \int_U \chi_{B_T\cap B_U}(\eta) dv \\
=&
\frac{1}{2\pi |b|}
|U|\chi_{B_T\cap B_U}(\eta).
\end{aligned}
\end{equation}

Consequently, we have
\begin{equation}
\begin{aligned}
\left\|
\tau_{B_U \times U}\tau_{B_T \times T}
\right\|_{HS}^2
&=
\int_{\Omega}\int_T
\frac{1}{2\pi |b|}
|U|\chi_{B_T\cap B_U}(\eta)
\,dx\,dP \\
&=
\frac{1}{2\pi |b|}
|U||T|P(B_T\cap B_U).
\end{aligned}
\end{equation}
	This completes the proof.
\end{proof}

By Lemma~\ref{lem}, it follows that
\begin{equation}
\left\|
\tau_{B_U \times U}\tau_{B_T \times T}
\right\|^2
\leq
\left\|
\tau_{B_U \times U}\tau_{B_T \times T}
\right\|_{HS}^2.
\end{equation}
Furthermore, applying Lemma~\ref{lem3}, we obtain
\begin{equation}
\label{th3}
\left\|
\tau_{B_U \times U}\tau_{B_T \times T}
\right\|^2
\leq
\frac{1}{2\pi |b|}
|U||T|P(B_T\cap B_U), \ \ b \neq 0.
\end{equation}

\begin{defn}
	The random signal $f(t,\eta)$ is said to be $\varepsilon_{B_T\times T}$-concentrated on $B_T\times T$ if there exists a function $f_1(t,\eta)$ supported on $B_T\times T$ such that
\begin{equation}
\label{df1}
	\|f(t,\eta)-f_1(t,\eta)\|_{\Omega}
	\leq
	\varepsilon_{B_T\times T}.
\end{equation}
	Similarly, $L_A(u,\eta)$ is said to be $\varepsilon_{B_U\times U}$-concentrated on $B_U\times U$ if there exists a function $f_2(u,\eta)$ supported on $B_U\times U$ such that
\begin{equation}
\label{df2}
	\left\|
	L_A(u,\eta)-f_2(u,\eta)
	\right\|_{\Omega}
	\leq
	\varepsilon_{B_U\times U}.
\end{equation}
\end{defn}
With the above preparation, we establish the DSUP in the LCT domain.
\begin{thm}
	\label{thm4}
	Assume that $\|f\|_{\Omega}=1$. If $f(t,\eta)$ is
	$\varepsilon_{B_T\times T}$-concentrated on $B_T\times T$ and
	$L_A(u,\eta)$ is
	$\varepsilon_{B_U\times U}$-concentrated on $B_U\times U$, then the DSUP is as follows
	\begin{equation}
\frac{1}{2\pi |b|}
	|U||T|P(B_T\cap B_U)
	\geq
	\left(
	1-\varepsilon_{B_T\times T}
	-\varepsilon_{B_U\times U}
	\right)^2, \ \ b \neq 0.
	\end{equation}
\end{thm}

\begin{proof}
	By \eqref{df1} and \eqref{df2}, we obtain
	\begin{equation}
	\begin{aligned}
	&\left\|
	f-\tau_{B_U\times U}\tau_{B_T\times T}f
	\right\|_{\Omega}\\
	\leq&
	\left\|
	f-\tau_{B_T\times T}f
	\right\|_{\Omega}
	+
	\left\|
	\tau_{B_T\times T}f
	-
	\tau_{B_U\times U}\tau_{B_T\times T}f
	\right\|_{\Omega}  \\
	\leq&
	\varepsilon_{B_T\times T}
	+
	\varepsilon_{B_U\times U}.
	\end{aligned}
	\end{equation}
	
Applying the reverse triangle inequality, we obtain
	\begin{equation}
	\|f\|_{\Omega}
	-
	\left\|
	\tau_{B_U\times U}\tau_{B_T\times T}f
	\right\|_{\Omega}
	\leq
	\left\|
	f-\tau_{B_U\times U}\tau_{B_T\times T}f
	\right\|_{\Omega}.
	\end{equation}
	Since $\|f\|_{\Omega}=1$, it follows that
	\begin{equation}
	\left\|
	\tau_{B_U\times U}\tau_{B_T\times T}f
	\right\|_{\Omega}
	\geq
	1-\varepsilon_{B_T\times T}
	-\varepsilon_{B_U\times U}.
	\end{equation}

	Then by \eqref{qf}, we have
	\begin{equation}
	\label{64}
	\left\|
	\tau_{B_U\times U}\tau_{B_T\times T}
	\right\|^2
	\geq
	\left(
	1-\varepsilon_{B_T\times T}
	-\varepsilon_{B_U\times U}
	\right)^2.
	\end{equation}
	
Applying \eqref{th3}, we obtain
	\begin{equation}
	\frac{1}{2\pi |b|}
	|U||T|P(B_T\cap B_U)
	\geq
	\left(
	1-\varepsilon_{B_T\times T}
	-\varepsilon_{B_U\times U}
	\right)^2.
	\end{equation}
	This completes the proof.
\end{proof}
\begin{corollary}
	 Assume that $\|f\|_{\Omega}=1$. If $f(t,\eta)$ is $\varepsilon_{B_T\times T}$-concentrated on $B_T\times T$ and $F_{\alpha}(\rho,\eta)$ is	$\varepsilon_{B_U\times U}$-concentrated on $B_U\times U$, then the DSUP in the fractional Fourier transform domain is as follows
	\begin{equation}
	\frac{1}{2\pi |\sin\alpha|}
	|U||T|P(B_T\cap B_U)
	\geq
	\left(
	1-\varepsilon_{B_T\times T}
	-\varepsilon_{B_U\times U}
	\right)^2, \ \ \sin\alpha \neq 0.
	\end{equation}
\end{corollary}
\begin{corollary}
Assume that $\|f\|_{\Omega}=1$. If $f(t,\eta)$ is $\varepsilon_{B_T\times T}$-concentrated on $B_T\times T$ and its Fourier transform (FT) $\hat{f}(\xi,\eta)$ is $\varepsilon_{B_U\times U}$-concentrated on $B_U\times U$, then the DSUP in the FT domain is as follows
	\begin{equation}
	\frac{|U||T|}{2\pi}
	P(B_T\cap B_U)
	\geq
	\left(
	1-\varepsilon_{B_T\times T}
	-\varepsilon_{B_U\times U}
	\right)^2 .
	\label{FTDS}
	\end{equation}
The DSUP in \eqref{DSFT} is formulated in the normalized FT domain, and its form is essentially consistent with the result in \eqref{FTDS}.
\end{corollary}

\begin{defn}
	Let $n(t,\eta)$ denote a random noise. The observed noisy random signal $r(t,\eta)$ is defined by
	\begin{equation}
	\label{r}
	r(t,\eta)
	=
	\begin{cases}
	f(t,\eta)+n(t,\eta), & (t,\eta)\in (B_T\times T)^c,\\
	0, & (t,\eta)\in B_T\times T.
	\end{cases}
	\end{equation}
\end{defn}
A stable recovery result for random signals in the LCT domain is presented in the following theorem.
\begin{thm}
Assume that $|U||T|P(B_T\cap B_U) < 2\pi |b|, b \neq 0$.
If there exists a linear operator $Q$ satisfying $
\|f-Qr\|_{\Omega}
\leq
C\|n\|_{\Omega}$,
where
$
C
\leq
\left(
1-
\sqrt{
	\frac{|U||T|P(B_T\cap B_U)}
	{2\pi |b|}
}
\right)^{-1},
$
then the random signal $f(t,\eta)$ can be recovered uniquely from $r(t,\eta)$.
\end{thm}

\begin{proof}
First, we prove uniqueness. Suppose that there exists another random signal $f^\star(t,\eta)$ that is also recovered from $r(t,\eta)$. Define the difference between the two signals by
	\begin{equation}
	l(t,\eta)=f(t,\eta)-f^\star(t,\eta).
	\end{equation}
	Since both $f(t,\eta)$ and $f^\star(t,\eta)$ coincide with the observation
	on $(B_T\times T)^c$, we have
	\begin{equation}
	l(t,\eta)=0,
	\qquad (t,\eta)\in (B_T\times T)^c .
	\end{equation}
	Hence, $l(t,\eta)$ is supported on $B_T\times T$. 
	Moreover, since both signals are frequency-limited to $B_U\times U$ in the LCT domain, it follows that
	\begin{equation}
	\tau_{B_U\times U}l
	=
	\tau_{B_U\times U}(f-f^\star)
	=
	f-f^\star
	=
	l.
	\end{equation}
	Thus, $l$ is also frequency-limited to $B_U\times U$ in the LCT domain.
	
	If $l(t,\eta)$ were not identically zero, by the Theorem \ref{thm4} to $l(t,\eta)$, we have
	\begin{equation*}
\frac{1}{2\pi |b|}
|U||T|P(B_T\cap B_U)
\geq
\left(
1-\varepsilon_{B_T\times T}
-\varepsilon_{B_U\times U}
\right)^2.
\end{equation*}
	This contradicts the assumption $|U||T|P(B_T\cap B_U) < 2 \pi |b| $.
	Therefore,
	\begin{equation}
	l(t,\eta)=0,
	\quad (t,\eta)\in\mathbb{R}\times\Omega,
	\end{equation}
	and hence $f^\star(t,\eta)=f(t,\eta)$ almost everywhere. This proves uniqueness.
	
	We next construct the recovery operator. By the \eqref{th3}, we obtain
	\begin{equation}
	\left\|
	\tau_{B_T\times T}\tau_{B_U\times U}
	\right\|
	=
	\left\|
	\tau_{B_U\times U}\tau_{B_T\times T}
	\right\|
	\leq
	\sqrt{
\frac{1}{2\pi |b|}
|U||T|P(B_T\cap B_U)}
\leq 1.
	\end{equation}
	Consequently, the operator
	$I-\tau_{B_T\times T}\tau_{B_U\times U}$ is invertible, and we define
	\begin{equation}
	Q=
	\left(
	I-\tau_{B_T\times T}\tau_{B_U\times U}
	\right)^{-1}.
	\end{equation}
	
	Since
	\begin{equation}
	\left(	I-\tau_{B_T\times T}\right) f(t,\eta) = \left(
	I-\tau_{B_T\times T}\tau_{B_U\times U}
	\right) f(t,\eta) ,
	\end{equation}
	we have
	\begin{equation}
	\begin{aligned}
	&f(t,\eta)-Qr(t,\eta)\\
	=&
	f(t,\eta)
	-
	Q(I-\tau_{B_T\times T})f(t,\eta)
	-
	Qn(t,\eta) \\[6pt]
	=&
	f(t,\eta)
	-
	\left(
	I-\tau_{B_T\times T}\tau_{B_U\times U}
	\right)^{-1}
	\left(
	I-\tau_{B_T\times T}\tau_{B_U\times U}
	\right)f(t,\eta)
	-
	Qn(t,\eta) \\[6pt]
	=&
	-Qn(t,\eta).
	\end{aligned}
\end{equation}
	Therefore, we obtain
	\begin{equation}
\begin{aligned}
	&\|f(t,\eta)-Qr(t,\eta)\|_{\Omega}\\
	=&
	\|Qn(t,\eta)\|_{\Omega} \\
	\leq&
	\|Q\|\,\|n(t,\eta)\|_{\Omega} \\
	\leq&
	\left(
	1-
	\left\|
	\tau_{B_T\times T}\tau_{B_U\times U}
	\right\|
	\right)^{-1}
	\|n(t,\eta)\|_{\Omega} \\
	\leq&
	\left(
	1-
	\left\|
	\tau_{B_T\times T}\tau_{B_U\times U}
	\right\|_{HS}
	\right)^{-1}
	\|n(t,\eta)\|_{\Omega} \\
	=&
	\left(
	1-
	\sqrt{
	\frac{1}{2\pi |b|}
		|U||T|P(B_T\cap B_U)}
	\right)^{-1}
	\|n(t,\eta)\|_{\Omega}.
	\end{aligned}
\end{equation}
\end{proof}

The following theorem further establishes the stability of random signal recovery under noisy conditions.
\begin{thm}
	Suppose that $|U||T|P(B_T\cap B_U) < 2 \pi |b|, b \neq 0 $,  and let the noise satisfy
	\begin{equation}
	\label{n}
	\|n(t,\eta)\|_{\Omega}^2
	\leq
	\left(
	1-\frac{1}{2\pi |b|}|T||U|P(B_T\cap B_U)
	\right)\frac{\delta}{4}.
	\end{equation}
If the reconstructed random signal $\widetilde{f}(t,\eta)$, obtained from the noisy observation $r(t,\eta)$, satisfies\textbf{}
	\begin{equation}
	\label{r-}
	\left\|
	r(t,\eta)
	-
	\tau_{(B_T\times T)^c}\widetilde{f}(t,\eta)
	\right\|_{\Omega}^2
	\leq
	\left(	1-\frac{1}{2\pi |b|}|T||U|P(B_T\cap B_U)\right)
		\frac{\delta}{4},
	\end{equation}
	then the reconstruction error satisfies
	\begin{equation}
	\left\|
	f(t,\eta)-\widetilde{f}(t,\eta)
	\right\|_{\Omega}^2
	\leq
	\delta.
	\end{equation}
\end{thm}

\begin{proof}
	First, we decompose the reconstruction error as
	\begin{equation}
	\begin{aligned}
	\left\|
	f(t,\eta)-\widetilde{f}(t,\eta)
	\right\|_{\Omega}^2
	&=
	\left\|
	\tau_{(B_T\times T)^c}
	\left(
	f(t,\eta)-\widetilde{f}(t,\eta)
	\right)
	\right\|_{\Omega}^2 \\
	&\quad+
	\left\|
	\left(
	I-\tau_{(B_T\times T)^c}
	\right)
	\left(
	f(t,\eta)-\widetilde{f}(t,\eta)
	\right)
	\right\|_{\Omega}^2 .
	\end{aligned}
	\end{equation}
	
	For the first term, by \eqref{r}, \eqref{n} and \eqref{r-}, we have
	\begin{equation}
	\label{4.41}
	\begin{aligned}
	&\left\|
	\tau_{(B_T\times T)^c}
	\left(
	f(t,\eta)-\widetilde{f}(t,\eta)
	\right)
	\right\|_{\Omega}^2 \\
	&=
	\left\|
	\tau_{(B_T\times T)^c}
	\left(
	r(t,\eta)-n(t,\eta)
	\right)
	-
	\tau_{(B_T\times T)^c}\widetilde{f}(t,\eta)
	\right\|_{\Omega}^2 \\[6pt]
	&\leq
	\left(
	\left\|
	r(t,\eta)
	-
	\tau_{(B_T\times T)^c}\widetilde{f}(t,\eta)
	\right\|_{\Omega}
	+
	\left\|
	n(t,\eta)
	\right\|_{\Omega}
	\right)^2 \\[6pt]
	&\leq
	\left(
	1-\frac{1}{2\pi |b|}|U||T|P(B_T\cap B_U)
	\right)\delta .
	\end{aligned}
	\end{equation}
	
	For the second term, since $I-\tau_{(B_T\times T)^c}=\tau_{B_T\times T}$, and by using \eqref{th3}, we obtain
	\begin{equation}
	\label{4.42}
	\begin{aligned}
	&\left\|
	\left(
	I-\tau_{(B_T\times T)^c}
	\right)
	\left(
	f(t,\eta)-\widetilde{f}(t,\eta)
	\right)
	\right\|_{\Omega}^2 \\
	&=
	\left\|
	\left(
	I-\tau_{(B_T\times T)^c}
	\right)
	\tau_{B_U\times U}
	\left(
	f-\widetilde{f}
	\right)(t,\eta)
	\right\|_{\Omega}^2 \\[6pt]
	&=
	\left\|
	\tau_{B_T\times T}\tau_{B_U\times U}
	\left(
	f-\widetilde{f}
	\right)(t,\eta)
	\right\|_{\Omega}^2 \\[6pt]
	&\leq
	\left\|
	\tau_{B_T\times T}\tau_{B_U\times U}
	\right\|^2
	\left\|
	f(t,\eta)-\widetilde{f}(t,\eta)
	\right\|_{\Omega}^2 \\[6pt]
	&\leq
	\frac{1}{2 \pi |b|}
	|T||U|P(B_T\cap B_U)
	\left\|
	f(t,\eta)-\widetilde{f}(t,\eta)
	\right\|_{\Omega}^2 .
	\end{aligned}
	\end{equation}
	
	Combining \eqref{4.41} and \eqref{4.42}, we have
	\begin{equation}
	\begin{aligned}
	\left\|
	f(t,\eta)-\widetilde{f}(t,\eta)
	\right\|_{\Omega}^2
	&\leq
	\left(
	1-\frac{1}{2\pi |b|}|U||T|P(B_T\cap B_U)
	\right)\delta \\
	&\quad+
\frac{1}{2\pi |b|}
	|T||U|P(B_T\cap B_U)
	\left\|
	f(t,\eta)-\widetilde{f}(t,\eta)
	\right\|_{\Omega}^2 .
	\end{aligned}
	\end{equation}
	Therefore, we get
	\begin{equation}
	\left\|
	f(t,\eta)-\widetilde{f}(t,\eta)
	\right\|_{\Omega}^2
	\leq
	\delta .
	\end{equation}
	This completes the proof.
\end{proof}

\section{Conclusion}
\label{sec5}
In this paper, we investigate uncertainty principles for random signals in the linear canonical transform (LCT) domain. By formulating the LCT of random signals on the probability space, we establish the corresponding Heisenberg and Donoho-Stark uncertainty principles. Both principles reveal that the uncertainty bounds depend jointly on the random signals and the parameters of the LCT, thereby providing additional degrees of freedom in controlling the balance between time and frequency localization.
Since the LCT provides a unified generalization of the Fourier transform and the fractional Fourier transform, the corresponding results in these two domains are also presented as special cases. The findings of this work are of both theoretical and practical significance, particularly in characterizing time-frequency localization and the recovery of random signals from incomplete or partially missing observations.

\section*{Acknowledgment}
This work was supported by grants from the National Natural Science Foundation of China [No. 62571042] and Natural Science Foundation of Beijing Municipality [No. 4242011].

\bibliographystyle{elsarticle-num}	
\bibliography{mybibfile}

@article{dang2023,
	title={Uncertainty principles for random signals},
	author={Dang, Pei and Li, Chuang and Mai, Wei Xiong and Pan, Wen Liang},
	journal={Appl. Math. Comput.},
	volume={444},
	pages={127833},
	year={2023},
	publisher={Elsevier}
}

@article{heisenberg,
	title={{\"U}ber den anschaulichen Inhalt der quantentheoretischen Kinematik und Mechanik},
	author={Heisenberg, Werner},
	journal={Zeitschrift f{\"u}r Physik},
	volume={43},
	number={3},
	pages={172--198},
	year={1927},
	publisher={Springer}
}

@article{dsu,
	title={Uncertainty principles and signal recovery},
	author={Donoho, David L and Stark, Philip B},
	journal={SIAM J. Appl. Math.},
	volume={49},
	number={3},
	pages={906--931},
	year={1989},
	publisher={SIAM}
}

@book{HUN,
	title={Time-frequency transforms for radar imaging and signal analysis},
	author={Chen, Victor C and Ling, Hao},
	year={2002},
	publisher={Artech house}
}

@article{DStwo,
	title={Two aspects of the {Donoho-Stark} uncertainty principle},
	author={Boggiatto, Paolo and Carypis, Evanthia and Oliaro, Alessandro},
	journal={J. Math. Anal. Appl.},
	volume={434},
	number={2},
	pages={1489--1503},
	year={2016},
	publisher={Elsevier}
}

@article{DSE,
	title = {Uncertainty principle in random quaternion domains},
	journal = {Digit. Signal Prog.},
	volume = {136},
	pages = {103988},
	year = {2023},
	issn = {1051-2004},
	author = {Pei Dang and Wei Xiong Mai and Wen Liang Pan},
	publisher={Elsevier}
}

@article{hw,
	title={On the {Heisenberg-Weyl} inequality},
	author={Rassias, John Michael},
	journal={J. Inequ. Pure \&  Appl. Math},
	volume={6},
	year={2005},
	number={1},
	pages={1--18}
}

@article{uncertaintyhj,
	title={Uncertainty principles for the offset linear canonical transform},
	author={Huo, Hai Ye},
	journal={Circuits Syst. Signal Process.},
	volume={38},
	number={1},
	pages={395--406},
	year={2019},
	publisher={Springer}
}

@article{uwlct,
	title={Uncertainty principles for the windowed offset linear canonical transform},
	author={Gao, Wen Biao and Li, Bing Zhao},
	journal={Int. J. Wavelets Multiresolut. Inf. Process.},
	volume={20},
	number={01},
	pages={2150042},
	year={2022},
	publisher={World Scientific}
}

@article{RandomFrFt,
	title = {Sampling random signals in a fractional {Fourier} domain},
	journal = {Signal Process.},
	volume = {91},
	number = {6},
	pages = {1394-1400},
	year = {2011},
	issn = {0165-1684},
	author = {Ran Tao and Feng Zhang and Yue Wang},
	publisher={Elsevier}
}

@article{xurandom,
	title={Analysis of {A}-stationary random signals in the linear canonical transform domain},
	author={Xu, Shui Qing and Feng, Li and Chai, Yi and He, Yi Gang},
	journal={Signal Process.},
	volume={146},
	pages={126--132},
	year={2018},
	publisher={Elsevier}
}

@article{dslp,
	title={{$L^{p}$}-{Donoho-Stark} principle for {Jacobi-Dunkl} transform},
	author={Safouane, Najat and Achak, Azzedine and Loualid, El Mehdi},
	journal={Sao Paulo J. Math. Sci.},
	volume={19},
	number={1},
	pages={9},
	year={2025},
	publisher={Springer}
}

@book{LCT,
	title={Linear canonical transforms},
	author={Healy, John J and Kutay, M Alper and Ozaktas, Haldun M and Sheridan, John T},
	journal={Springer Series in Optical Sciences},
	pages={453},
	year={2016},
	publisher={Springer}
}

@article{FrFtLp,
	title={{$L^{p}$}-type {Heisenberg-Pauli-Weyl} uncertainty principles for fractional {Fourier} transform},
	author={Chen, Xuan and Dang, Pei and Mai, Wei Xiong},
	journal={Appl. Math. Comput.},
	volume={491},
	pages={129236},
	year={2025},
	publisher={Elsevier}
}

@article{zhang2025H,
	title={Standard {Heisenberg’s} uncertainty principles of {Cohen’s} class time-frequency distribution with specific kernels},
	author={Zhang, Zhi Chao},
	journal={J. Pseudo-Differ. Oper. Appl.},
	volume={16},
	number={3},
	pages={67},
	year={2025},
	publisher={Springer}
}

@article{xulctup,
	title={Three uncertainty relations for real signals associated with linear canonical transform},
	author={Xu, Guan Lei and Xiao Tong, W and Xiao Gang, X},
	journal={IET Signal Process.},
	volume={3},
	number={1},
	pages={85--92},
	year={2009},
	publisher={IET}
}

@article{mairandomlct,
	title={Uncertainty principles for linear canonical transform},
	author={Mai, Wei Xiong and Dang, Pei and Pan, Wen Liang and Chen, Xuan},
	journal={J. Math. Anal. Appl.},
	volume={548},
	number={2},
	pages={129414},
	year={2025},
	publisher={Elsevier}
}

@article{2026Suncertainty,
	title={On Uncertainty Principles Associated with a Generalized Stockwell Transform in Offset Fractional {Fourier} Transform},
	author={Bachtiar, Nasrullah and Nur, Andi Tenri Ajeng and Rivai, Muklas},
	journal={J. Math. Anal. Appl.},
	pages={130686},
	year={2026},
	publisher={Elsevier}
}

@article{2026Buncertainty,
	title={The uncertainty principles for the multidimensional {Fourier}-{Bessel} transform},
	author={Xu, Hui and Zhao, Xing Yu and Wei, Long Ben},
	journal={J. Math. Anal. Appl.},
	pages={130414},
	year={2026},
	publisher={Elsevier}
}

@article{tuan2010donoho,
	title={{Donoho}-{Stark} and {Paley}-{Wiener} theorems for the {G}-transform},
	author={Tuan, Vu Kim and Yakubovich, Semyon},
	journal={Adv. Appl. Math.},
	volume={45},
	number={1},
	pages={108--124},
	year={2010},
	publisher={Elsevier}
}

@article{peng20252p,
	title={The 2p order {Heisenberg}-{Pauli}-{Weyl} uncertainty principles related to the offset linear canonical transform},
	author={Peng, Jia Yin and Li, Bing Zhao},
	journal={Digit. Signal Prog.},
	pages={105746},
	year={2025},
	publisher={Elsevier}
}

@article{randomqt,
	Author = {Qian, Tao},
	Title = {Sparse representations of random signals},
	Journal = {Math. Meth. Appl. Sci.},
	Year = {2022},
	Volume = {45},
	Number = {8},
	Pages = {4210-4230},
	Month = {MAY 30},
	EarlyAccessDate = {DEC 2021},
	ISSN = {0170-4214},
	EISSN = {1099-1476},
	Unique-ID = {WOS:000734618200001},
}

@article{frftrandom,
	Author = {Torres, Rafael and Torres, Edmanuel},
	Title = {Fractional Fourier Analysis of Random Signals and the Notion of
	{$\alpha$}-Stationarity of the {Wigner-Ville} Distribution},
	Journal = {IEEE Trans. Signal Process.},
	Year = {2013},
	Volume = {61},
	Number = {6},
	Pages = {1555-1560},
	Month = {MAR},
	ISSN = {1053-587X},
	EISSN = {1941-0476},
	ORCID-Numbers = {Torres, Rafael/0000-0002-5739-4413},
	Unique-ID = {WOS:000315965500022},
}

\end{document}